\documentclass[5pt]{article}
\usepackage{amsmath}
\usepackage[latin1]{inputenc}
\usepackage{amsmath,amsthm,amssymb}
\usepackage{graphicx}
\usepackage[all,poly,knot]{xy}

\usepackage{color}

\renewcommand{\theequation}{\thesection.\arabic{equation}}

\newtheorem{defn}{Definition}[section]
\newtheorem{lem}{Lemma}[section]
\newtheorem{thm}{Theorem} [section]
\newtheorem{prop}{Proposition} [section]
\newtheorem{exmp}{Example} [section]
\newtheorem{coro}{Corollary}[section]
\newtheorem{rem}{Remark}[section]

\theoremstyle{remark}

\newcommand{\Tr}{{\rm Tr}}
\newcommand{\gf}{ {{\mathbb F}} }
\numberwithin{equation}{section}

\title{Algebraic Structure of Permutational Polynomials over $\gf_{q^n}$
\thanks{Supported By NSF of China No. 12171163 }}

\author{
Pingzhi Yuan\thanks{ P. Yuan is with School of  of Mathematical Science, South China Normal University,  Guangzhou 510631, China (email: yuanpz@scnu.edu.cn).}}

    \date{}
\begin{document}
\baselineskip15pt \maketitle
\renewcommand{\theequation}{\arabic{section}.\arabic{equation}}
\catcode`@=11 \@addtoreset{equation}{section} \catcode`@=12

    \begin{abstract}In this paper, we propose  a new  algebraic structure of  permutation polynomials over $\gf_{q^n}$. As an application of this new algebraic structure, we give some classes of new PPs over $\gf_{q^n}$ and answer an open problem in Charpin and  Kyureghyan~\cite{CK09}.
\end{abstract}


\section{Introduction}

\,\,\, Let $q$ be a prime power, $\mathbb{F}_q$ be the finite field of order $q$, and $\mathbb{F}_q[x]$
be the ring of polynomials in a single indeterminate $x$ over $\mathbb{F}_q$.   We know that every map from $\mathbb{F}_q$ to $\mathbb{F}_q$ can be viewed as a polynomial in the set~$\left(\mathbb{F}_q[x]/(x^q-x), +, \circ\right)$, where the two operations are the addition and the composition of two polynomials modulo $x^q-x$.  A polynomial
$f \in\mathbb{F}_q[x]/(x^q-x)$ is called a {\em permutation polynomial} (PP for short) of $\mathbb{F}_q$ if it induces
a bijective map from $\mathbb{F}_q$ to itself. The  unique polynomial denoted by $f^{-1}(x)$ in~$\mathbb{F}_q[x]/(x^q-x)$ such that $f\circ f^{-1}(x)=f^{-1}\circ f(x)=x$ is called the compositional inverse of $f(x)$, where $x$ is the unity in $\mathbb{F}_q[x]/(x^q-x)$.

Permutation polynomials with simple algebraic forms  over finite fields  and their compositional inverses have wide applications in coding
theory \cite{Ding13, DZ14, LC07}, cryptography \cite{RSA, SH}, combinatorial
design theory \cite{DY06}, and other
areas of mathematics and engineering \cite{LN86, Mull}. For more  on
PPs and their compositional inverses, we refer to \cite{AGW11, LQW19, MP13, NLQW21, TW17, W24, WL13, WY24, Y22, YD11, YD14}.

 Let
 $n$ be a positive integer, a polynomial of the form
 $$ L(x) =\sum_{i=0}^{n-1}a_ix^{q^i}\in\gf_{q^n}[x]$$
 is called a $q$-polynomial (or a linear polynomial) over $\gf_{q^n}$, and is a permutation polynomial of $\gf_{q^n}$ if
 and only if $L(x)$ only has the root $0$ in $\gf_{q^n}$ \cite[Theorem 7.9]{LN86}. The trace function $\Tr^n_1(\cdot)$ (or $\Tr(\cdot)$ for short) from $\gf_{q^n}$
to $\gf_q$ is defined by
$$\Tr(x)=\Tr^n_1(x) = x + x^q+\dots+x^{q^{n-1}}, \quad x\in\gf_{q^n}.$$

Recall that the field $\gf_{q^n}$  may be viewed as a vector space of dimension $n$ over $\gf_q$. Therefore,  PPs over  $\gf_{q^n}$ may have some algebraic structures. The motivation of the present paper is to provide some algebraic structures of PPs over  $\gf_{q^n}$.

Let $u_1, \dots, u_n$ and $v_1, \dots, v_n$  be a dual pair of ordered bases  of $\gf_{q^n}$ over $\gf_q$, and let $f(x)\in\gf_{q^n}[x]$ be a polynomial. Then for any $c\in\gf_{q^n}$, $f(c)=u_1a_1+\dots+u_na_n, a_i\in\gf_q, 1\le i\le n$, we define
$$f_i: \gf_{q^n}\to\gf_{q}\quad \mbox{by}\quad f_i(c)=a_i, 1\le i\le n.$$
It is easy to see that $f_i(x)=\Tr(v_if(x))\in\gf_{q^n}[x], 1\le i\le n$ are well-defined and are uniquely determined by $f(x)$ and the base $v_1, \dots, v_n$. Moreover, we have
$$f(x)=u_1f_1(x)+\dots+u_nf_n(x)=u_1\Tr(v_1f(x))+\dots+u_n\Tr(v_nf(x)).$$
Throughout the paper, we view a polynomial  as an element of $\mathbb{F}_q[x]/(x^q-x)$, and we fix the above notations.

In 2008, Zhou~\cite{Z08} showed that  there
 are exactly $(q^n- 1)(q^n-q)\dots (q^n-q^{n-1})$ different linear permutation polynomials. 
In 2011, Yuan and Zeng~\cite{YZ11} proved the following Theorem \cite[Theorems 1.1 and 1.2]{YZ11}.

\begin{thm}\label{thYZ}1. Let $\{\omega_1, \omega_2, \dots, \omega_n\}$ be any given basis of $\gf_{q^n}$  over $\gf_{q}$, and let $L(x) = \sum_{i=0}^{n-1}a_ix^{q^i}$ be a
 linear polynomial over $\gf_{q^n}$. Then there are n elements $\theta_1, \theta_2, \dots, \theta_n\in\gf_{q^n}$ such that
$$ L(x) =\Tr(\theta_1x)\omega_1 +\dots+\Tr(\theta_nx)\omega_n.$$
 Moreover, $L(x)$ is a permutation polynomial if and only if $\{\theta_1, \theta_2, \dots, \theta_n\}$ is a basis of $\gf_{q^n}$  over $\gf_{q}$. And there
 are exactly $(q^n- 1)(q^n-q)\dots (q^n-q^{n-1})$ different linear permutation polynomials, in fact they must be of the form as above  with a basis $\{\theta_1, \theta_2, \dots, \theta_n\}$.

 2. Let $\{\theta_1, \theta_2, \dots, \theta_n\}$ be any given basis of $\gf_{q^n}$  over $\gf_{q}$, and let
$$ L(x) =\Tr(\theta_1x)\omega_1 +\dots+\Tr(\theta_nx)\omega_n, \omega_i\in\gf_{q^n}.$$
 Then $L(x)$ is a permutation polynomial if and only if $\{\omega_1, \omega_2, \dots, \omega_n\}$ is a basis of $\gf_{q^n}$  over $\gf_{q}$.\end{thm}

The main purpose of the present paper is to prove the following much more general result.

\begin{thm}\label{thnew0}If $f(x)\in\gf_{q^n}[x]$ is a PP over $\gf_{q^n}$ and $h_i(x)\in\gf_q[x], 1\le i\le n$ are maps from $\gf_{q}$ to $\gf_{q}$, then the polynomial
$$F(x)=a_1h_1(f_1(x))+\dots+a_nh_n(f_n(x)), \quad a_i\in\gf_{q^n}, 1\le i\le n,$$
where $f_i(x)=\Tr(v_if(x)), 1\le i\le n$ are defined above, is a PP over $\gf_{q^n}$ if and only if the following two conditions hold:

(1) $a_1, \dots, a_n$ is a basis of $\gf_{q^n}$ over $\gf_q$;

(2) $h_i(x)\in\gf_q[x], 1\le i\le n$ are PPs over $\gf_{q}$.

Moreover, if $f^{-1}(x)$ and $h_i^{-1}(x), 1\le i\le n$ are the compositional inverses of $f(x)$ and $h_i(x), 1\le i\le n$, respectively, and $b_1, \dots, b_n$ is the dual base of $a_1, \dots, a_n$, then we have
$$F^{-1}(x)=f^{-1}\left(\sum_{i=1}^nu_ih_i^{-1}(\Tr(b_ix))\right).$$
\end{thm}
As a consequence of the above two theorems, we have
\begin{coro}\label{coro1}Let $\theta_1, \theta_2, \dots, \theta_n$ be $n$ elements  of $\gf_{q^n}$  over $\gf_{q}$, and let
$$ F(x) =a_1\Tr(\theta_1x)^{m_1} +\dots+a_n\Tr(\theta_nx)^{m_n}, a_i\in\gf_{q^n}, m_i\in \mathbb{N}.$$
 Then $F(x)$ is a PP over $\gf_{q^n}$ if and only if the following three conditions hold

 (i) $\gcd(m_1\dots m_n, q-1)=1$,

 (ii) $a_1, a_2, \dots, a_n$ is a basis of $\gf_{q^n}$  over $\gf_{q}$.

 (iii) $\theta_1, \theta_2, \dots, \theta_n$ is a basis of $\gf_{q^n}$  over $\gf_{q}$.
 \end{coro}
The remaining part of this paper is organized as follows. Section 2 focuses on the properties of the algebraic structure of PPs over $\gf_{q^n}$. In Section 3, we will prove the main theorem of this paper and present some corollaries. As an application of the results in Section 2, in Section 3,  we obtain some results on the PPs of the shape $G(X) + \gamma\Tr(H(X))$,
where $G(X), H(X)\in\gf_{q^n}, \gamma\in\gf_{q^n}$. This answers an open problem in Charpin and  Kyureghyan~\cite{CK09}.
\section{Some lemmas and propositions}
To begin with, we have the following simple and important result, which will be used later.
\begin{lem}\label{le2.1} Let $f(x)\in\gf_{q^n}[x]$ and let
$A_i=\{f_i(x)=\Tr(v_if(x)), x\in\gf_{q^n}\}, 1\le i\le n$. Then
$$\sharp Im f(x)\le\prod_{i=1}^n\sharp A_i,$$
where $\sharp S$ denotes the cardinality of the set $S$.\end{lem}

Now we have the following proposition:
\begin{prop}\label{prop1} Let the notations be as before. Then
$f(x)\in\gf_{q^n}[x]$ is a PP over $\gf_{q^n}$ if and only if  $\sharp f^{-1}(u_1a_1+\dots u_na_n)=\sharp \cap_{i=1}^nf_i^{-1}(a_i)=1$ for any $(a_1, \dots, a_n)\in\gf_q^n$.\end{prop}

\begin{proof} Since  $u_1, \dots, u_n$ is a basis of $\gf_{q^n}$ over $\gf_q$, for any $(a_1, \dots, a_n)\in\gf_q^n$, we have
$$\sharp f^{-1}(u_1a_1+\dots u_na_n)=\sharp \cap_{i=1}^nf_i^{-1}(a_i).$$ Therefore, $f(x)\in\gf_{q^n}[x]$ is a PP over $\gf_{q^n}$ if and only if  $\sharp f^{-1}(u_1a_1+\dots u_na_n)=\sharp \cap_{i=1}^nf_i^{-1}(a_i)=1$ for any $(a_1, \dots, a_n)\in\gf_q^n$. This proves the proposition.\end{proof}

To state the next result, we need the following definition of $n$-to-1 mapping.

\begin{defn} Let $f$ be a mapping from one finite set $A$ to another finite set $B$. Then $f$
 is called $n$-to-1 if one of the following two cases holds:

 (1) if $n |\sharp A$, and for any $b\in B$, it has either $n$ or $0$ preimages of $f$ in $A$;

 (2) if $n\not|\sharp A$, and for almost all $b\in B$, it has either $n$ or $0$ preimages of $f$ in $A$, and for
the only one exception element, it has exactly $\sharp A\pmod{n}$ preimages. \end{defn}

\begin{prop}\label{prop2} If $f(x)\in\gf_{q^n}[x]$ is a PP over $\gf_{q^n}$, then the following two statements hold:

(1) $A_i=\gf_q, 1\le i\le n$;

(2) For $t\in\{1, 2, \dots, n\}$, the maps $\varphi_t(x): \gf_{q^n}\to \gf_q^t$  defined by
$$ \varphi_t(x)=(f_1(x), \dots, f_t(x))=(\Tr(v_1f(x)), \dots, \Tr(v_tf(x))),\quad 1\le t\le n,$$
 are $q^{n-t}$-to-1 mappings.\end{prop}
\begin{proof} (1) Since $u_1, \dots, u_n$ is a basis of $\gf_{q^n}$ over $\gf_q$, by Lemma \ref{le2.1}, we have
$$\sharp Im f(x)\le\prod_{i=1}^n\sharp A_i.$$
Now the condition  $f(x)\in\gf_{q^n}[x]$ is a PP over $\gf_{q^n}$ implies that $\sharp Im f(x)=q^n$. Hence $\sharp A_i=q$ and $A_i=\gf_q$.

As for (2), for any $(a_1, \dots, a_t)\in\gf_q^t$, there are precisely $q^{n-t}$ elements $y=u_1y_1+\dots+u_ny_n\in\gf_{q^n}$ such that $y_1=a_1, \dots, y_t=a_t$ and $y_{t+1}, \dots, y_n\in\gf_{q}$. Hence
\begin{align*}\varphi_t^{-1}(a_1, \dots, a_t)&=\cap_{i=1}^tf_i^{-1}(a_i)\\
&=\{x\in\gf_{q^n}, \,\, f(x)=u_1a_1+\dots+u_ta_t+b_{t+1}u_{t+1}+\dots+b_nu_n, b_i\in\gf_q\}\\
&=\cup_{b_j\in\gf_q}\left(\cap_{i=1}^tf_i^{-1}(a_i)\cap_{j=t+1}^nf_j^{-1}(b_j)\right).\\
\end{align*}
It follows from Proposition \ref{prop1} that $\sharp \varphi_t^{-1}(a_1, \dots, a_t)=q^{n-t}$, as desired.
 \end{proof}
 Let $L$ be the set of all $q$-linear PPs over $\gf_{q^n}$, i.e.,
 $$L=\{l(x)=a_0x+a_1x^q+\dots+a_{n-1}x^{q^{n-1}}, \,\, a_i\in\gf_{q^n}, 1\le i\le n, l(x) \,\, \mbox{is a PP over}\,\, \gf_{q^n}\}.$$

 \begin{prop}\label{prop3}  If $f(x)\in\gf_{q^n}[x]$ is a PP over $\gf_{q^n}$ and $b_1, \dots, b_n\in\gf_{q^n}$, then the polynomial
 $$g(x)=b_1\Tr(v_1f(x))+\dots+b_n\Tr(v_nf(x))$$
 is a PP over $\gf_{q^n}$ if and only if $b_1, \dots, b_n$ is a basis of $\gf_{q^n}$ over $\gf_q$.\end{prop}
 \begin{proof}If $l(x)=a_0x+a_1x^q+\dots+a_{n-1}x^{q^{n-1}}\in L$, then
 $$l(f(x))=a_0f(x)+a_1f(x)^q+\dots+a_{n-1}F(x)^{q^{n-1}}=w_1\Tr(v_1f(x))+\dots+w_n\Tr(v_nf(x)), w_i\in\gf_{q^n}$$
 is also a PP over $\gf_{q^n}$. On the other hand, if $l_1(x)\in L$ and $l_1(x)\ne l(x)$, then $l_1(f(x))\ne l(f(x))$. Note that the number of $b_1, \dots, b_n\in\gf_{q^n}$ such that $b_1, \dots, b_n$ is a basis of $\gf_{q^n}$ over $\gf_q$ is equal to $(q^n-1)(q^n-q)\dots(q^n-q^{n-1})=\sharp L$. Hence $g(x)=b_1\Tr(v_1f(x))+\dots+b_n\Tr(v_nf(x))$
 is a PP over $\gf_{q^n}$ if and only if $b_1, \dots, b_n$ is a basis of $\gf_{q^n}$ over $\gf_q$. This proves the proposition.\end{proof}

The following two results of Yuan \cite{Yuan24} built an algebraic frame for computing the compositional inverse of PP.
\begin{lem}{\rm(\cite[Theorem 2.4]{Yuan24}) }\label{th2.1} A polynomial $f(x)\in\mathbb{F}_q[x]$ is a PP if and only if for any maps $\psi_i, \, i=1, \ldots, t, t\in\mathbb{N}$ (we also denote them as $\psi_i(x)\in\mathbb{F}_q[x]/(x^q-x)$) such that
$F(\psi_1(x), \ldots, \psi_t(x))=x$ for some polynomial $F(x_1, \ldots, x_t)\in \mathbb{F}_q[x_1, \ldots, x_t]$, there exists a polynomial $G(x_1, \ldots, x_t)\in \mathbb{F}_q[x_1, \ldots, x_t]$ satisfies $G(\psi_1(f(x)), \ldots, \psi_t(f(x)))=x$. Moreover, if $f(x)$ is a PP, then
$$f^{-1}(x)=G(\psi_1(x), \ldots, \psi_t(x)),$$
where $f^{-1}(x)$ denotes  the compositional inverse of $f(x)$.\end{lem}

\begin{lem}\label{th2.2} {\rm(\cite[Theorem 2.5]{Yuan24}) }Let $q$ be a prime power and $f(x)$ be a polynomial over $\mathbb{F}_q$. Then $f(x)$ is a PP if and only if there exist nonempty finite subsets $S_i, i=1, \ldots, t$  of $\mathbb{F}_q$ and maps $\psi_i: \mathbb{F}_q\to S_i, i=1, \ldots, t$  such that $\psi_i\circ f=\varphi_i, i=1, \ldots, t$ and $x=F(\varphi_1(x), \ldots, \varphi_t(x))$, where $F(x_1, \ldots, x_t)\in\mathbb{F}_q[x_1, \ldots, x_t]$. Moreover, we have
$$ f^{-1}(x)=F(\psi_1(x), \ldots, \psi_t(x)).$$\end{lem}

As an application of Lemma \ref{th2.1}, we have

\begin{prop}\label{prop4} Let $u_1, \dots, u_n$ and $v_1, \dots, v_n$  be a dual pair of ordered bases  of $\gf_{q^n}$ over $\gf_q$, and let $f(x)\in\gf_{q^n}[x]$ be a polynomial. Then $f(x)$ is a PP over $\gf_{q^n}$ if and only if there exists $F(x_1, \ldots, x_n)\in\mathbb{F}_{q^n}[x_1, \ldots, x_n]$ such that
$$x=F(f_1(x), \ldots, f_n(x)).$$\end{prop}
\begin{proof}Let $\psi_i(x)=\Tr(v_ix), 1\le i\le n$. Then
$$\varphi_i(x)=\psi_i(f(x))=\Tr(v_if(x))=f_i(x), \quad 1\le i\le n.$$
Note that $u_1, \dots, u_n$ and $v_1, \dots, v_n$  are a dual pair of ordered bases  of $\gf_{q^n}$ over $\gf_q$, we have
$$x=u_1\Tr(v_1x)+\dots+u_n\Tr(v_nx)=u_1\psi_1(x)+\dots+u_n\psi_n(x).$$
By Lemma \ref{th2.1}, $f(x)$ is a PP over $\gf_{q^n}$ if and only if there exists $F(x_1, \ldots, x_n)\in\mathbb{F}_{q^n}[x_1, \ldots, x_n]$ such that
$x=F(\varphi_1(x), \dots, \varphi_n(x))=F(f_1(x), \ldots, f_n(x))$. This proves the proposition.\end{proof}
The following proposition is an inverse result of Proposition \ref{prop2} (2).
\begin{prop}\label{pro2.5}Let $t, n>0$ be  positive integers with $t\le n$ and $g_1(x), \dots, g_t(x)$ be maps from $\gf_{q^n}$ to $\gf_q$ such that the map
$$\varphi_t: \gf_{q^n} \to \gf_q^t,\quad \varphi_t(x)=(g_1(x), \dots, g_t(x))$$
is a $q^{n-t}$-to-1 mapping. Then, for any basis $u_1, \dots, u_n$ of $\gf_{q^n}$ over $\gf_q$, there exists a PP $f(x)$ over $\gf_{q^n}$ satisfying
$$f(x)=u_1g_1(x)+\dots+u_tg_t(x)+u_{t+1}f_{t+1}(x)+\dots+u_nf_n(x)$$
and $f_i(x), t+1\le i\le n$ are some maps from $\gf_{q^n}$ to $\gf_q$.\end{prop}

\begin{proof} By the assumptions, we have
$$\gf_{q^n}=\cup_{(a_1, \dots, a_t)\in\gf_q^t}\varphi^{-1}((a_1, \dots, a_t)).$$
For $x\in\varphi^{-1}((a_1, \dots, a_t))$, we define
$$f(x)=u_1a_1+\dots+u_ta_t+b_{t+1}u_{t+1}+\dots+b_nu_n, \quad b_i\in\gf_q, t+1\le i\le n$$
and $f(x)\ne f(y)$ if $x, y\in \varphi^{-1}((a_1, \dots, a_t))$ and $x\ne y$. This is possible since
$$\sharp \varphi^{-1}((a_1, \dots, a_t))=\sharp\{b_{t+1}u_{t+1}+\dots+b_nu_n, \quad b_i\in\gf_q, t+1\le i\le n\}=q^{n-t}.$$
It follows that the polynomial $f(x)$ defined as above is injective, hence a PP over $\gf_{q^n}$. Moreover, we have $f(x)=u_1g_1(x)+\dots+u_tg_t(x)+u_{t+1}f_{t+1}(x)+\dots+u_nf_n(x)$, where $f_i(x)=b_i, t+1\le i\le n$ are some maps from $\gf_{q^n}$ to $\gf_q$. This completes the proof.\end{proof}

\begin{rem} It follows from the proof of Proposition \ref{pro2.5} that the number of such PPs is equal to $(q^{n-t}!)^{q^t}$ since $\sharp \varphi^{-1}((a_1, \dots, a_t))=q^{n-t}$ and there are $q^t$ choices for such $(a_1, \dots, a_t)$. Moreover, there exists a PP and a basis $v_1, \dots, v_n$ of $\gf_{q^n}$ over $\gf_q$ such that $g_i(x)=\Tr(v_if(x)), 1\le i\le t$.
\end{rem}
\section{ Proof of Theorem 1.2}
In this section, we will prove Theorem 1.2.

{\bf Proof of Theorem 1.2:}

\begin{proof} Since $F(x)$ is a PP and $F(x)\in\langle u_1, \dots, u_n\rangle$ for any $x\in\gf_{q^n}$, we have $\langle u_1, \dots, u_n\rangle=\gf_{q^n}$, so $u_1, \dots, u_n$ is a basis of $\gf_{q^n}$ over $\gf_q$ and $h_i(x)\in\gf_q[x], 1\le i\le n$ are PPs over $\gf_{q}$. This proves the necessity.

Now we prove the sufficiency. For any $(a_1, \dots, a_n)\in\gf_q^n$, since $u_1, \dots, u_n$ is a basis of $\gf_{q^n}$ over $\gf_q$ and $h_i(x)\in\gf_q[x], 1\le i\le n$ are PPs over $\gf_{q}$, we have
$$\cap_{i=1}^n\left(h_i\circ f_i\right)^{-1}(a_i)=\cap_{i=1}^nf_i^{-1}\left(h_i^{-1}(a_i)\right).$$
By Proposition \ref{prop1} and the fact that $f(x)$ is a PP over $\gf_{q^n}$, we have
$$\sharp\cap_{i=1}^n\left(h_i\circ f_i\right)^{-1}(a_i)=1.$$
It follows that $F(x)$ is a PP over $\gf_{q^n}$ by Proposition \ref{prop1} again.

Let $\psi_i(x)=\Tr(b_ix), 1\le i\le n$, where $b_1, \dots, b_n$ is the dual base of $a_1, \dots, a_n$. Then we have
$$\varphi_i(x)=\psi_i(F(x))=h_i(f_i(x))=h_i(\Tr(v_if(x))), 1\le i\le n.$$
Hence $\Tr(v_if(x))=h_i^{-1}(\varphi_i(x)), 1\le i\le n$. Since $u_1, \dots, u_n$ is the dual base of $v_1, \dots, v_n$, we have
$$f(x)=\sum_{i=1}^nu_i\Tr(v_if(x))=\sum_{i=1}^nu_ih_i^{-1}(\varphi_i(x)).$$
It follows that $x=f^{-1}\left(\sum_{i=1}^nu_ih_i^{-1}(\varphi_i(x))\right)$, by Lemma \ref{th2.2}, we get
$$F^{-1}=f^{-1}\left(\sum_{i=1}^nu_ih_i^{-1}(\Tr(b_ix))\right).$$
This proves the theorem.\end{proof}

\begin{rem} It follows from Theorem \ref{thnew0} that the number of such PPs is $(q^n-1)\dots(q^n-q^{n-1})(q!)^n$.\end{rem}

As  immediately consequences of Theorem \ref{thnew0}, we have the following  corollary.

\begin{coro}\label{coro0}If $f(x)\in\gf_{q^n}[x]$ is a PP over $\gf_{q^n}$ and $m_i\in\mathbb{N}, 1\le i\le n$ are positive integers, then the polynomial
$$F(x)=a_1(f_1(x))^{m_1}+\dots+a_n(f_n(x))^{m_n}, \quad a_i\in\gf_{q^n}, 1\le i\le n$$
is a PP over $\gf_{q^n}$ if and only if the following two conditions hold:

(1) $\gcd(m_1\dots m_n, q-1)=1$;

(2) $a_1, \dots, a_n$ is a basis of $\gf_{q^n}$ over $\gf_q$.
\end{coro}
\begin{exmp}Let $q$ be a prime power with $q\equiv3\pmod{4}$. Then $x^2+1$ is irreducible over $\gf_q$, so $\frac{1}{2}, \frac{i}{2}(i=\sqrt{-1})$ and $\frac{1}{2}, -\frac{i}{2}$ are a dual pair of ordered bases  of $\gf_{q^2}$ over $\gf_q$. Hence $\Tr(ix)=ix+(ix)^q=i(x-x^q)$, and
$$2x=\Tr(x)-i\Tr(ix)=(x+x^q)+(x-x^q).$$
By Corollary \ref{coro0}, for an odd positive integer $m$,
$$\left(\Tr(x)\right)^m-i\left(\Tr(ix)\right)^m=(x+x^q)^m+(x-x^q)^m$$
is a PP over $\gf_{q^2}$ if and only if $\gcd(m, q-1)=1$. \end{exmp}
For example, let $m=3$, we have
$$x^3+3x^{1+2q}$$
is a PP over $\gf_{q^2}$ if and only if $q\equiv3\pmod{4}$ and $\gcd(3, q-1)=1$. Let $m=5$, we have
$$x^5+10x^{3+2q}+5x^{1+4q}$$
is a PP over $\gf_{q^2}$ if and only if $q\equiv3\pmod{4}$ and $\gcd(5, q-1)=1$. Let $m=11$, we have
$$x^{11}+55x^{9+2q}+330x^{7+4q}+462x^{5+6q}+165x^{3+8q}+11x^{1+10q}$$
is a PP over $\gf_{q^2}$ if and only if $q\equiv3\pmod{4}$ and $\gcd(11, q-1)=1$. In particular,
$$x^{11}+x^{9+2\cdot3^m}-x^{1+10\cdot3^m}$$
is a PP over $\gf_{3^{2m}}$ if and only if $\gcd(m, 10)=1$ since $\gcd(11, 3^m-1)=1$ if and only if $5\not|m$.

If $m=p^r+2<q$, $q\equiv3\pmod{4}$ and $\gcd(p^r+2, q-1)=1$, then we have
$$\frac{1}{2}\left((x+x^q)^m+(x-x^q)^m\right)=x^{p^r+2}+x^{p^r+2q}+2x^{1+(p^r+1)q}$$
is a PP over $\gf_{q^2}$.

\begin{exmp}Let $q$ be a prime power with such that $-3$ is not a square in $\gf_q$, i.e. $q\equiv5\pmod{12}$ or $q\equiv 11\pmod{12}$ or $q=2^t$ with odd positive integer $t$. Then $x^2+x+1$ is irreducible over $\gf_q$, let $\omega$ be a root of $x^2+x+1$ in $\gf_{q^2}$ and $\omega^3=1$. If $q\equiv5, 11\pmod{12}$, then $1, \omega$ is a basis  of $\gf_{q^2}$ over $\gf_q$ and $\Tr(\omega x)=\omega x+(\omega x)^q=\omega(x+\omega x^q)$, and by Corollary 1.1,
 for an odd positive integer $m$,
$$\left(\Tr(x)\right)^m-\omega\left(\Tr(\omega x)\right)^m=(x+x^q)^m-\omega\left(\omega x+\omega^2 x^q\right)^m$$
is a PP over $\gf_{q^2}$ if and only if $\gcd(m, q-1)=1$. \end{exmp}

If $p=2$ is the characteristic of $\gf_q$ and $r$ an odd positive integer, $m=2^r+1$. Then $3|m$, $\gcd(2^r+1, q)=\gcd(2^r+1, 2^t-1)=1$ and
$$\left(\Tr(x)\right)^m+\omega\left(\Tr(\omega x)\right)^m=\omega^2x^{2^r+1}+\omega^2x^{q(2^r+1)}+\omega x^{2^r+q}$$
is a PP over $\gf_{q^2}$. That is, $x^{2^r+1}+x^{q(2^r+1)}+\omega^2 x^{2^r+q}$ is a PP over $\gf_{q^2}$.

\section{PPs of the shape $G(x) + \gamma\Tr(H(x))$}

In this section,  we consider the polynomials of the shape
\begin{equation}\label{eq1}F (x) = G(x) + \gamma\Tr(H(x)),\end{equation}
where $G(x), H(x)\in\gf_{q^n}[x], \gamma\in\gf_{q^n}$. In \cite{CK09}, Charpin and Kyureghyan characterized and constructed such permutation polynomials. Moreover, they proposed the following problem.

{\bf  Problem.} {\rm \cite[Open Problem 1]{CK09}} Characterize a class of permutation polynomials of type (\ref{eq1}), where $G(X)$ is neither
a permutation nor a linearized polynomial.

By Proposition \ref{pro2.5}, we have the following corollary.

\begin{coro}\label{coro3} If $H(x)\in[x]$ is a polynomial such that $\Tr(H(x))$ is a $q^{n-1}$-to-1 mapping, $\gamma\in\gf_{q^n}^\star$, then the number of PPs of the type
$$u_1f_1(x) +\dots +u_{n-1}f_{n-1}(x)+\gamma \Tr(H(x)),$$
where $f_i(x), 1\le i\le n-1$ are maps from $\gf_{q^n}$ to $\gf_q$, is equal to $(q^{n-1}!)^q$.\end{coro}

We have

\begin{thm}\label{th3} If $H(x)\in\gf_{q^n}[x]$ is a polynomial such that $\Tr(H(x))$ is not a constant map, $\gamma\in\gf_{q^n}^\star$, then there are at least $(q^{n-1}!)^q-(q^n-q)\dots(q^n-q^{n-1})$ PPs of the form $u_1f_1(x) +\dots +u_{n-1}f_{n-1}(x)+\gamma \Tr(H_1(x))+\gamma \Tr(H(x))$ such that $G(x)=u_1f_1(x) +\dots +u_{n-1}f_{n-1}(x)+\gamma \Tr(H_1(x))$ is neither a permutation nor a linearized polynomial.\end{thm}
\begin{proof} We first consider the case when $\Tr(H(x))$ is a $q^{n-1}$-to-1 mapping, by Corollary \ref{coro3}, there exists a PP of the type
$$f(x)=u_1f_1(x) +\dots +u_{n-1}f_{n-1}(x)+\gamma \Tr(H(x)).$$
Then $u_1, \dots, u_{n-1}, \gamma$ is a basis of $\gf_{q^n}$ over $\gf_q$.  Let $v_1, \dots, v_{n-1}, v_n$ be the dual base of $u_1, \dots, u_{n-1}, \gamma$. Then $\Tr(H(x))=\Tr(v_nf(x))$.

Since the number of the $q$-linear PPs of the form $u_1\Tr(\theta_1x) +\dots +u_{n-1}\Tr(\theta_{n-1}x)+\gamma \Tr(\theta_nx)$ is $(q^n-q)\dots(q^n-q^{n-1})$, so there are at least $(q^{n-1}!)^q-(q^n-q)\dots(q^n-q^{n-1})$ PPs of the form $u_1f_1(x) +\dots +u_{n-1}f_{n-1}(x)+\gamma \Tr(H(x))$ such that $G(x)=u_1f_1(x) +\dots +u_{n-1}f_{n-1}(x)$ is neither a permutation nor a linearized polynomial.

 Next we consider the case when $\Tr(H(x))$ is neither a $q^{n-1}$-to-1 mapping and nor a constant.
Put
$$M(a)=\{x\in\gf_{q^n}, \, H(x)\in\Tr^{-1}(a)\}=H^{-1}\left(\Tr^{-1}(a)\right), \,\, a\in\gf_q.$$
If $\sharp M(a)=q^{n-1}$ for any $a\in\gf_q$, then $\Tr(H(x))$ is $q^{n-1}$-to-1, and we are done.  Therefore we may assume that $\sharp M(a)>q^{n-1}$ for some $a\in\gf_q$. Recall that $M(a)\ne\gf_{q^n}$, so there exists an element $b\in\gf_q$ such that $b\ne a$ and $\sharp M(b)>0$.

Set
$$M(a)=M_1(a)\uplus M_2(a), \, \sharp M_1(a)=q^{n-1}.$$

$$M(b)=\begin{cases}
M_1(b)\uplus M_2(b), \, \sharp M_1(b)=q^{n-1}, & \, \mbox{if}\,\, \sharp M(b)\geq^{n-1},\\
M_1(b), & \,\, \mbox{otherwise}.\\
\end{cases}$$
For any $a\in\gf_q$, we choose an element $w(a)\in\gf_{q^n}$ with $\Tr(w(a))=a$. Let $H_2(x)\in\gf_{q^n}[x]$  be the polynomial with $H_2(x)=a$ for any $x\in\Tr^{-1}(a)$. Define
$$H_1(x)=\begin{cases}
w(0), & x\in M_1(a)\cup M_1(b),\\
H_2(x)-H(x), & x\in\gf_{q^n}\backslash\left(M_1(a)\cup M_1(b)\right).
\end{cases}$$
Now it is easy to check that $\Tr(H_1(x)+H(x))=\Tr(H_2(x)), x\in\gf_{q^n}$ and $\Tr(H_2(x))$ is a $q^{n-1}$-to-1 mapping. Since $H_1^{-1}\left(\Tr^{-1}(0)\right)\supseteq M_1(a)\cup M_1(b)$ and
$\sharp(M_1(a)\cup M_1(b))>q^{n-1}$, which implies that $\Tr(H_1(x))$ is not a $q^{n-1}$-to-1 mapping.

Since $\Tr(H_2(x))$ is a $q^{n-1}$-to-1 mapping, we can construct at least $(q^{n-1}!)^q-(q^n-q)\dots(q^n-q^{n-1})$ PPs of the form $u_1f_1(x) +\dots +u_{n-1}f_{n-1}(x)+\gamma \Tr(H_2(x))=u_1f_1(x) +\dots +u_{n-1}f_{n-1}(x)+\gamma \Tr(H_1(x))+\gamma\Tr(H(x))$ such that $G(x)=u_1f_1(x) +\dots +u_{n-1}f_{n-1}(x)+\gamma\Tr(H_1(x))$ is neither a permutation nor a linearized polynomial. This completes the proof.\end{proof}

\noindent\textbf{Acknowledgements}\textit{ We are very grateful to Professor Shaoshi Chen, Professor Qiang Wang, Dr. Lianjun Bian and Shuang Li  for their  comments and suggestions that have been a great  help in improving the quality of this paper.}


\begin{thebibliography}{10}


\bibitem{AGW11}A. Akbary, D. Ghioca, and Q. Wang, On constructing permutations
of finite fields, Finite Fields  Appl.,  17(2001), 51-67.


\bibitem{CK09} P. Charpin, G. Kyureghyan,  When does $G(x)+\gamma \Tr(H(x))$ permute $\mathbb{F}_{p^n}$? Finite Fields Appl. 15 (2009), no. 5, 615-632.




\bibitem{Ding13} C. Ding, Cyclic codes from some monomials and trinomials, SIAM J. Discrete Math. 27(2013),1977-1994.

\bibitem{DY06} C. Ding and J. Yuan, A family of skew Hadamard difference sets,
J. Comb. Theory, Ser. A 113 (2006),  1526-1535.

\bibitem{DZ14} C. Ding and Z. Zhou, Binary cyclic codes from explicit polynomials over $GF(2^m)$, Discrete Math. 321(2014), 76-89.


\bibitem{LC07} Y.
Laigle-Chapuy, Permutation polynomials and applications to coding
theory, Finite Fields Appl. 13 (2007),  58--70.


\bibitem{LQW19} K. Li, L. Qu, and Q. Wang, Compositional inverses of permutation
polynomials of the form $x^rh(x^s)$ over finite fields, Cryptogr. Commun.,
11(2019),279-298.


\bibitem{MP13} G. L. Mullen, D. Panario,  Handbook of Finite Fields. CRC Press, Boca Raton, 2013.



\bibitem{LN86} R. Lidl, H. Niederreiter,
Introduction to finite fields and their applications,  Cambridge
University Press, Cambridge, 1986.


\bibitem{Mull} G. L. Mullen, Permutation polynomials over finite fields, In: Proc.
Conf. Finite Fields and Their Applications, Lecture Notes in Pure and Applied
Mathematics, vol. 141, Marcel Dekker, 1993, 131--151.






\bibitem{NLQW21} T. Niu, K. Li, L. Qu, and Q. Wang,Finding compositional inverses of permutations from the AGW creterion, IEEE Trans. Inf. Theory, 67(2021), 4975-4985.


\bibitem{RSA} R. L. Rivest, A. Shamir, and L. M. Adelman, A method for obtaining digital signatures and public-key cryptosystems,
Comm. ACM 21 (1978), 120--126.

\bibitem{SH}
J. Schwenk and K. Huber,
Public key encryption and digital signatures based on permutation polynomials,
Electronic Letters 34 (1998), 759--760.


\bibitem{TW17} A. Tuxanidy and Q. Wang, Compositional inverses and complete mappings over finite fields, Discrete Appl. Math., 217(2017), 318-329.




\bibitem{W24} Q. Wang, A survey of compositional inverses of permutation polynomials over finite fields. Designs,codes and crytography, https://doi.org/10.1007/
s10623-024-01436-4.



\bibitem{WL13} B. Wu and Z. Liu, Linearized polynomials over finite fields revisited,
Finite Fields  Appl., 22(2013), 79-100.



\bibitem{WY24} D. Wu, P. Yuan, Permutation polynomials and their compositional inverses over finite fields by a local method. Des. Codes Cryptogr. 92 (2024), no. 2, 267-276.

\bibitem{Y22} P. Yuan, Compositional Inverses of AGW-PPs, Adv. Math. Comm.,16 (2022), no. 4, 1185-1195.


\bibitem{Yuan24} P. Yuan, Local method for compositional inverses of permutation polynomials. Comm. Algebra 52 (2024), no. 7, 3070-3080


\bibitem{YD11} P. Yuan and C. Ding, Permutation polynomials over finite fields from a
powerful lemma, Finite Fields  Appl., 17(2011), 560-574.

\bibitem{YD14} P. Yuan, C. Ding, Further results on permutation polynomials over finite fields. Finite Fields Appl. 27 (2014), 88-103.

\bibitem{YZ11} P. Yuan and X. Zeng, A note on linear permutation polynomials, Finite Fields Appl. 17 (2011) 488-491.


\bibitem{Z08} K. Zhou, A remark on linear permutation polynomials, Finite Fields Appl. 14 (2008) 532-536.












\end{thebibliography}
\end{document}